\input ./style/arxiv-vmsta.cfg
\documentclass[numbers,compress,v1.0.1]{vmsta}
\usepackage{vtexbibtags}

\usepackage{algorithm}
\usepackage{algorithmic}

\volume{6}
\issue{4}
\pubyear{2019}
\firstpage{495}
\lastpage{513}
\aid{VMSTA145}
\doi{10.15559/19-VMSTA145}
\articletype{research-article}

\DeclareMathOperator{\prvtex}{\mathsf{P}}
\DeclareMathOperator{\Mvtex}{\mathsf{E}}
\DeclareMathOperator{\Dvtex}{Var}
\DeclareMathOperator{\cov}{Cov}
\startlocaldefs
\newcommand{\bGamma}{{\boldsymbol \Gamma}}
\newcommand{\bXi}{{\boldsymbol \Xi}}
\newcommand{\bSigma}{{\boldsymbol \Sigma}}
\newcommand{\bxi}{{\boldsymbol \xi}}
\newcommand{\bydef}{\stackrel{\text{\rm def}}{=}}
\newcommand{\weak}{\stackrel{\text{\rm W}}{\longrightarrow}}

\newtheorem{thm}{Theorem}
\newtheorem{lem}{Lemma}
\newtheorem{prop}{Proposition}

\theoremstyle{definition}
\newtheorem{experiment}{Experiment}

\allowdisplaybreaks
\ifdefined\HCode 
\def\index#1{}
\else 
\fi
\endlocaldefs

\begin{document}

\begin{frontmatter}
\pretitle{Research Article}

\title{Jackknife covariance matrix estimation for observations from mixture}

\author{\inits{R.}\fnms{Rostyslav}~\snm{Maiboroda}\thanksref{cor1}\ead[label=e1]{mre@univ.kiev.ua}}
\author{\inits{O.}\fnms{Olena}~\snm{Sugakova}\ead[label=e2]{sugak@univ.kiev.ua}}
\thankstext[type=corresp,id=cor1]{Corresponding author.}
\address{\institution{Taras Shevchenko National University of Kyiv}, Kyiv, \cny{Ukraine}}



\markboth{R. Maiboroda, O. Sugakova}{Jackknife covariance matrix estimation for observations from mixture}

\begin{abstract}
A general jackknife estimator for the asymptotic covariance of moment
estimators is considered in the case when the sample is taken from a
mixture with varying concentrations of components. Consistency of the
estimator is demonstrated. A fast algorithm for its calculation is
described. The estimator is applied to construction of confidence sets
for regression parameters in the linear regression with errors in
variables. An application to sociological data analysis is considered.
\end{abstract}

\received{\sday{20} \smonth{5} \syear{2019}}
\revised{\sday{7} \smonth{9} \syear{2019}}
\accepted{\sday{14} \smonth{10} \syear{2019}}
\publishedonline{\sday{7} \smonth{11} \syear{2019}}

\end{frontmatter}

\section{Introduction}%
\label{SectIntr}

Finite Mixture Models\index{Finite Mixture Models (FMM)} (FMM) are widely used in the analysis of
biological, economic and sociological data. For a comprehensive survey
of different statistical techniques based on FMMs, see
\cite{MLR}. Mixtures with Varying Concentrations (MVC) is a subclass of
these models in which the mixing probabilities are not constant, but
vary for different observations (see \cite{Maiboroda2003,MS2012-2}).

In this paper we consider application of the jackknife technique\index{jackknife ! technique} to the
estimation of asymptotic covariance matrix (the covariance matrix for
asymptotically normal estimator, ACM) in the case when the data are
described by the MVC model.\index{MVC model} The jackknife\index{jackknife} is a well-known resampling
technique usually applied to i.i.d. samples (see \xch{Section}{section} 5.5.2 in
\cite{Shao}, \xch{Chapter}{chapter} 4 in \cite{Wolter}, \xch{Chapter}{chapter} 2 in
\cite{ShaoTu}). On the jackknife estimates\index{jackknife ! estimates} of variance for censored and
dependent data, see \cite{TG}. Its modification to the case of the MVC
model\index{MVC model} in which the observations are still independent but not
identically distributed needs some efforts.

We obtained a general theorem on consistency of the jackknife estimators\index{jackknife ! estimator}
for ACM for moment estimators in the MVC models\index{MVC model} and apply this result to
construct confidence sets for regression coefficients\index{regression coefficients} in linear
errors-in-variables models for MVC data. On general errors-in-variables
models, see \cite{Branham,Cheng,Kukush}. The model and the
estimators for the regression coefficients\index{regression coefficients} considered in this paper was
proposed in \cite{MNS}, where the asymptotic normality of these
estimates is shown.

The rest of the paper is organized as follows. In Section
\ref{SecMixt} we introduce the MVC model\index{MVC model} and describe the estimation
technique for these models based on weighted moments. In Section
\ref{SecMain} the jackknife estimates\index{jackknife ! estimates} for the ACM are introduced and
conditions of their consistency formulated. Section \ref{SecAlg} is
devoted to the algorithm of fast computation of the jackknife estimates.\index{jackknife ! estimates}
In Section \ref{SecErV} we apply the previous results to construct
confidence sets for linear regression coefficients\index{linear regression coefficients} in
errors-in-variables models with MVC. In Section \ref{SecSimul} results
of simulations are presented. In Section \ref{SecEit} we present results
of application of the proposed technique to analyze sociological data.
Proofs are placed in Section \ref{SecProof}. Section \ref{SecConcl}
contains concluding remarks.

\section{Mixtures with varying concentrations}%
\label{SecMixt}

We consider a dataset in which each observed subject $O$ belongs to one
of $M$ subpopulations (mixture components). The number $\kappa (O)$ of
the population which $O$ belongs to is unknown. We observe $d$ numeric
characteristics of $O$ which form the vector $\xi (O)=(\xi ^{1}(O),
\dots ,\xi ^{d}(O))^{T}\in \mathbb{R}^{d}$ of observable variables. The
distribution of $\xi (O)$ may depend on the component $\kappa (O)$:
\begin{equation*}
F_{\xi }^{(m)}(A)=\prvtex \{\xi (O)\in A\ |\ \kappa (O)=m \},\quad  m=1,
\dots ,M,
\end{equation*}
where $A$ is any Borel subset of $\mathbb{R}^{d}$.

We observe variables of $n$ independent subjects $\xi _{j}=\xi (O_{j})$.
The probability to obtain $j$-th subject from $m$-th component
\begin{equation*}
p_{j}^{(m)}=\prvtex \{\kappa (O_{j})=m \}
\end{equation*}
can be considered as the concentration of the $m$-th component in the
mixture when the $j$-th observation was made. The concentrations are
known and can vary for different observations.

So, the distribution of $\xi _{j}$ is described by the model of mixture
with varying concentrations:
%
\begin{equation}
\label{Eq1s}
\prvtex \{\xi _{j}\in A \}=\sum _{m=1}^{M} p_{j}^{(m)} F^{(m)}_{\xi }(A)
\end{equation}
We will denote by
\begin{equation*}
\mu ^{(m)}=\Mvtex^{(m)}[\xi ]=\Mvtex [\xi (O)\ |\ \kappa (O)=m]
=
\int _{\mathbb{R}^{d}}xF^{(m)}_{\xi }(dx)
\end{equation*}
the vector of theoretical first moments of the $m$-th component
distribution. In what follows, $\cov^{(m)}[\xi ]$ means the covariance
of $\xi (O)$ for the $m$-th component, $\Dvtex^{(m)}[\xi ^{l}]$ means the
variance of $\xi ^{l}(O)$ for this component and so on.

To estimate $\mu ^{(k)}$ by observations $\xi _{1}$, \dots , $\xi _{n}$ one
can use the weighted sample mean
%
\begin{equation}
\label{Eq0}
\bar{\xi }^{(k)}_{;n}=\bar{\xi }^{(k)}=\sum _{j=1}^{n} a_{j}^{(k)}\xi
_{j},
\end{equation}
where $a_{j}^{(k)}=a_{j;n}^{(k)}$ are some weights dependent on
components' concentrations, but not on the observed $\xi _{j}=\xi _{j;n}$.
(In what follows we denote by the subscript $;n$ that the corresponding
quantity is considered for the sample size $n$. In most cases this
subscript is dropped to simplify notations.)

To obtain unbiased estimates in (\ref{Eq0}) one needs to select the
weights satisfying the assumption
%
\begin{equation}
\label{Eq1}
\sum _{j=1}^{n} a_{j}^{(k)}p_{j}^{(m)}=
\begin{cases}
1
& \text{ if } k=m,
\\
0
& \text{ if } k\neq m.
\end{cases}
\end{equation}
Let us denote
\begin{gather*}
\bXi =(\xi _{1}^{T}, \dots ,\xi _{n})^{T}
=
\left (
\begin{array}{ccc}
\xi _{1}^{1}& \ldots &\xi _{1}^{d}
\\
\vdots &\ddots &\vdots
\\
\xi _{n}^{1}& \ldots &\xi _{n}^{d}
\end{array}
\right ),
\\
\mathbf{a}=\left (
\begin{array}{ccc}
a_{1}^{(1)} & \ldots & a_{1}^{(M)}
\\
\vdots & \ddots & \vdots
\\
a_{n}^{(1)} & \ldots & a_{n}^{(M)}
\end{array}
\right )
\quad \text{ and }\quad
\mathbf{p}=\left (
\begin{array}{ccc}
p_{1}^{(1)} & \ldots & p_{1}^{(M)}
\\
\vdots & \ddots & \vdots
\\
p_{n}^{(1)} & \ldots & p_{n}^{(M)}
\end{array}
\right ).
\end{gather*}
Then
$\mathbf{p}_{\centerdot }^{(m)}=(p_{1}^{(m)},\dots ,p_{n}^{(m)})^{T}$,
$\mathbf{p}_{j}^{\centerdot }=(p_{j}^{(1)},\dots ,p_{j}^{(M)})^{T}$ and
the same notation is used for the matrix $\mathbf{a}$.

In this notation the unbiasedness condition (\ref{Eq1}) reads
%
\begin{equation}
\label{Eq2}
\mathbf{a}^{T}\mathbf{p}=\mathbb{E},
\end{equation}
where $\mathbb{E}$ means the $M\times M$ unit matrix.

There can be many choices of $\mathbf{a}$ satisfying (\ref{Eq2}). In
\cite{Maiboroda2003,MS2012-2} minimax weights are
considered defined by\footnote{In fact, in
\cite{Maiboroda2003} and \cite{MS2012-2} the weights are defined
as $\tilde{\mathbf{a}}=n\mathbf{a}$ and $\bar{\xi }^{(k)}=\frac{1}{n}
\sum _{j=1}^{n} \tilde{a}_{j}^{(k)}\xi _{j}$. In this paper we adopt
notation which allows to simplify formulas for fast estimator
calculation in Section \ref{SecAlg}.}
%
\begin{equation}
\label{Eq3}
\mathbf{a}=\mathbf{p}\bGamma ^{-1},
\end{equation}
where
\begin{equation*}
\bGamma =\bGamma _{;n}=\mathbf{p}^{T}\mathbf{p}
\end{equation*}
is the Gram matrix of the set of concentration vectors $\mathbf{p}
_{\centerdot }^{(1)}$, \dots , $\mathbf{p}_{\centerdot }^{(M)}$. In what
follows, we assume that these vectors are linearly independent, so
$\det \bGamma >0$ and $\bGamma ^{-1}$ exists. See
\cite{MS2012-2} on the optimal properties of the estimates for
concentration distributions based on the minimax weights (\ref{Eq3}).

To describe the asymptotic behavior of $\bar{\xi }^{(k)}_{;n}$ as
$n\to \infty $, we will calculate its covariance matrix.

Notice that
\begin{equation*}
\cov [\xi _{j}]=\Mvtex [\xi _{j}\xi _{j}^{T}]-\Mvtex [\xi _{j}]\Mvtex [\xi
_{j}]^{T}=\sum _{m=1}^{M}p_{j}^{(m)}\bSigma ^{(m)}
-\sum _{m,l=1}^{M} p
_{j}^{(m)}p_{j}^{(l)}\mu ^{(m)}(\mu ^{(l)})^{T},
\end{equation*}
where $\bSigma ^{(m)}=\cov^{(m)}[\xi ]=\cov [\xi (O)\ |\ \kappa (O)=m]$.
So,
\begin{equation*}
n\cov \bar{\xi }^{(k)}=
\sum _{m=1}^{M}\langle (\mathbf{a}^{(k)})^{2}
\mathbf{p}^{(m)}\rangle _{;n}\bSigma ^{(m)}
-
\sum _{m,l=1}^{M} \langle
\xch{(\mathbf{a}}{\mathbf{a}}^{(k)})^{2}\mathbf{p}^{(m)}\mathbf{p}^{(l)}\rangle _{;n}\mu
^{(m)}(\mu ^{(l)})^{T},
\end{equation*}
where
\begin{equation*}
\langle (\mathbf{a}^{(k)})^{2}\mathbf{p}^{(m)}\rangle _{;n}=n\sum _{j=1}
^{n} (a_{j}^{(k)})^{2}p_{j}^{(m)},\qquad  \langle (\mathbf{a}^{(k)})^{2}
\mathbf{p}^{(m)}\mathbf{p}^{(l)}\rangle _{;n}=n\sum _{j=1}^{n} (a_{j}
^{(k)})^{2}p_{j}^{(m)}p_{j}^{(l)}.
\end{equation*}
Assume that the limits
%
\begin{equation}
\label{EqLim}
\langle \xch{(\mathbf{a}}{\mathbf{a}}^{(k)})^{2}\mathbf{p}^{(m)}\mathbf{p}^{(l)}
\rangle _{\infty }\bydef \lim _{n\to \infty }\langle \xch{(\mathbf{a}}{\mathbf{a}}^{(k)})^{2}
\mathbf{p}^{(m)}\mathbf{p}^{(l)}\rangle _{;n}
\end{equation}
exist. Then the limits
\begin{equation*}
\langle (\mathbf{a}^{(k)})^{2}\mathbf{p}^{(m)}\rangle _{\infty }\bydef
\lim _{n\to \infty }\langle (\mathbf{a}^{(k)})^{2}\mathbf{p}^{(m)}
\rangle _{;n}
\end{equation*}
exist also, since due to (\ref{Eq1}) we have $\sum _{l=1}^{M} p_{j}
^{l}=1$ for all $j$.

So, under this assumption,
%
\begin{equation}
\label{EqSigma}
n\cov [\bar{\xi }^{(k)}]\to \bSigma _{\infty }\quad \text{ as } n\to \infty,
\end{equation}
where
\begin{equation*}
\bSigma _{\infty }\bydef
 \sum _{m=1}^{M}\langle (\mathbf{a}^{(k)})^{2}\mathbf{p}^{(m)}
\rangle _{\infty }\bSigma ^{(m)}
-
\sum _{m,l=1}^{M} \langle \xch{(\mathbf{a}}{\mathbf{a}}
^{(k)})^{2}\mathbf{p}^{(m)}\mathbf{p}^{(l)}\rangle _{\infty }\mu ^{(m)}(
\mu ^{(l)})^{T}.
\end{equation*}
%
\begin{thm}
\label{Th1}
Assume that:
\begin{enumerate}
\item $\frac{1}{n}\bGamma _{;n}\to \bGamma _{\infty }$ as $n\to \infty $
and $\det \bGamma _{\infty }>0$.
\item Assumption (\ref{EqLim}) holds.
\item $\Mvtex^{(m)}[\|\xi \|^{2}]<\infty $ for all $m=1,\dots ,M$.
\end{enumerate}

Then
\begin{equation*}
\sqrt{n}(\bar{\xi }^{(k)}_{;n}-\mu ^{(k)})\weak N(0,\bSigma ^{(k)}
_{\infty }).
\end{equation*}
\end{thm}
This theorem is a simple corollary of Theorem 4.3 in
\cite{MS2012-2}.

\section{Jackknife estimation\index{jackknife ! estimation} of ACM of moment estimators}%
\label{SecMain}

In what follows, we will consider unknown parameters of the component
distribution $F^{(k)}_{\xi }$, which can be represented in the form
%
\begin{equation}
\label{Eq4}
\vartheta =\vartheta ^{(k)}=H(\mu ^{(k)}),
\end{equation}
where $H:\mathbb{R}^{d}\to \mathbb{R}^{q}$ is some known function. A
natural estimator for such parameter by the sample $\xi _{1}$, \dots , $
\xi _{n}$ is
%
\begin{equation}
\label{Eq5}
\hat{\vartheta }=\hat{\vartheta }_{;n}^{(k)}=H(\bar{\xi }^{(k)}_{;n}).
\end{equation}
Then asymptotic behavior of this estimator is described by the following
theorem.
%
\begin{thm}
\label{Th2}
In assumptions of Theorem \ref{Th1}, if $H$ is continuously
differentiable in some~neighborhood of $\mu ^{(k)}$, then
\begin{equation*}
\sqrt{n}(\hat{\vartheta }_{;n}^{(k)}-\vartheta ^{(k)})\weak N(0,
\mathbf{V}_{\infty }),
\end{equation*}
where
%
\begin{gather}
\label{Eq6}
\mathbf{V}_{\infty }=\mathbf{V}_{\infty }^{(k)}=\mathbf{H}'(\mu ^{(k)})\bSigma
_{\infty }^{(k)}(\mathbf{H}'(\mu ^{(k)}))^{T},
\\
\mathbf{H}'=\left (
\begin{array}{ccc}
\frac{\partial H^{1}}{\partial \mu ^{1}}&\ldots &\frac{\partial H^{1}}{
\partial \mu ^{d}}
\\
\vdots & \ddots & \vdots
\\
\frac{\partial H^{q}}{\partial \mu ^{1}}&\ldots &\frac{\partial H^{q}}{
\partial \mu ^{d}}
\end{array}
\right ).
\nonumber
\end{gather}
\end{thm}
This theorem is a simple implication from our Theorem \ref{Th1}
and
Theorem 3 in Section 5, Chapter 1 of \cite{Borovkov}.

So, $\mathbf{V}_{\infty }$ defined by (\ref{Eq6}) is the ACM of the
estimator $\hat{\vartheta }^{(k)}$ (the covariance matrix of the limit
normal distribution of the normalized difference between the estimator
and the estimated parameter). If it was known one could use it to
construct tests for hypotheses on $\vartheta ^{(k)}$ or to derive
confidence set for $\vartheta ^{(k)}$. In fact, for most estimators the
ACM is unknown. Usually some estimate of $\mathbf{V}_{\infty }$ is used
to replace its true value in statistical algorithms.

The jackknife\index{jackknife} is one of most general techniques of ACM estimation. Let
$\hat{\vartheta }$ be any estimator of $\vartheta $ by the data
$\xi _{1}$, \dots , $\xi _{n}$:
\begin{equation*}
\hat{\vartheta }=\hat{\vartheta }(\xi _{1},\dots ,\xi _{n}).
\end{equation*}
Consider estimates of the same form which are calculated by all
observations without one
\begin{equation*}
\hat{\vartheta }_{i-}=\hat{\vartheta }(\xi _{1},\dots ,\xi _{i-1},\xi
_{i+1},\dots ,\xi _{n}).
\end{equation*}
Then the jackknife estimator\index{jackknife ! estimator} for $\mathbf{V}_{\infty }$ is defined by
%
\begin{equation}
\label{Eq6a}
\hat{\mathbf{V}}_{;n}=\hat{\mathbf{V}}_{;n}^{(k)}=n\sum _{i=1}^{n}(
\hat{\vartheta }_{i-}-\hat{\vartheta })(\hat{\vartheta }_{i-}-
\hat{\vartheta })^{T}.
\end{equation}
In our case $\hat{\vartheta }=H(\bar{\xi }^{(k)})$, so
%
\begin{equation}
\label{Eq7}
\hat{\vartheta }_{i-}=H(\bar{\xi }^{(k)}_{i-}),
\end{equation}
where
%
\begin{equation}
\label{Eq7a}
\bar{\xi }^{(k)}_{j-}=\sum _{j\neq i}a_{ji-}^{(k)}\xi _{j}.
\end{equation}
Here $\mathbf{a}_{i-}=(a_{ji-}^{(m)},j=1,\dots ,n,\ m=1,\dots ,M )$ is
the minimax weights matrix calculated by the matrix $\mathbf{p}_{i-}$
of concentrations of all observations except $i$-th one. That is,
$\mathbf{p}_{i-}=(p_{1}^{\centerdot },\dots ,p_{i-1}^{\centerdot },0,p
_{i+1}^{\centerdot },\dots ,p_{n}^{\centerdot })^{T}$,
%
\begin{gather}
\label{Eq8}
\mathbf{a}_{i-}=\mathbf{p}_{i-}\bGamma _{i-}^{-1},
\\
\label{Eq9}
\bGamma _{i-}=\mathbf{p}_{i-}^{T}\mathbf{p}_{i-}.
\end{gather}
Notice that $0$ is placed at the $i$-th row of $\mathbf{p}_{i-}$ as a
placeholder only, to preserve the numbering of the rows in $
\mathbf{p}_{i-}$ and $\mathbf{a}_{i-}$, which corresponds to the
numbering of subjects in the sample.
%
\begin{thm}
\label{Th3}
Let $\vartheta $ be defined by (\ref{Eq4}), $\hat{\vartheta }$ by
(\ref{Eq5}), $\mathbf{V}_{\infty }$ by (\ref{Eq6}) and $\hat{V}_{;n}
^{(k)}$ by (\ref{Eq6a})--(\ref{Eq9}). Assume that:
\begin{enumerate}
\item $H$ is twice continuously differentiable in some neighborhood of
$\mu ^{(k)}$.
\item There exists some $\alpha >4$, such that $\Mvtex [\|\xi (O)\|^{
\alpha }\ |\ \kappa (O)=m]<\infty $ for all $m=1,\dots ,M$.
\item $\frac{1}{n}\bGamma _{;n}\to \bGamma _{\infty }$ as $n\to \infty $
and $\det \bGamma _{\infty }>0$.
\item Assumption (\ref{EqLim}) holds.
\end{enumerate}

Then $\hat{\mathbf{V}}_{;n}^{(k)}\to \mathbf{V}_{\infty }^{(k)}$.
\end{thm}
For \textbf{proof} see \xch{Section}{section}~\ref{SecProof}.

\section{Fast calculation algorithm for jackknife estimator\index{jackknife ! estimator}}%
\label{SecAlg}

Direct calculation of $\hat{\mathbf{V}}_{;n}$ by (\ref{Eq6a})--(\ref{Eq9})
needs $\sim Cn^{2}$ elementary operations. Here we consider an algorithm
which reduces the computational complexity to $\sim Cn$ operations
(linear complexity).

Notice that $\bGamma _{i-}=\bGamma -\mathbf{p}_{i}^{\centerdot }(
\mathbf{p}_{i}^{\centerdot })^{T}$. So
%
\begin{equation}
\label{Eq9a}
(\bGamma _{i-})^{-1}=\bGamma ^{-1}+\frac{1}{1-h_{i}}
\bGamma ^{-1}
\mathbf{p}_{i}^{\centerdot }(\mathbf{p}_{i}^{\centerdot })^{T}\bGamma
^{-1},
\end{equation}
where
%
\begin{equation}
\label{Eq9b}
h_{i}=(\mathbf{p}_{i}^{\centerdot })^{T}\bGamma ^{-1}\mathbf{p}_{i}
^{\centerdot }.
\end{equation}
(Formula (\ref{Eq9a}) can be demonstrated directly by checking
$\bGamma _{i-}^{-1}\bGamma _{i-}=\mathbb{E}$. It is also a corollary to
the Serman--Morrison--Woodbury formula, see A.9.4 in \cite{Seber}.)

Let us denote
\begin{equation*}
\bar{\bxi }_{i-}=
\left (
\begin{array}{ccc}
\bar{\xi }_{i-}^{1(1)}& \ldots &\bar{\xi }_{i-}^{d(1)}
\\
\vdots & \ddots &\vdots
\\
\bar{\xi }_{i-}^{1(M)}& \ldots &\bar{\xi }_{i-}^{d(M)}
\end{array}
\right )
=
((\bar{\xi }_{i-}^{(1)})^{T},\dots ,(\bar{\xi }_{i-}^{(M)})^{T})^{T}
\end{equation*}
and
%
\begin{equation}
\label{Eqbarbxi}
\bar{\bxi }=((\bar{\xi }^{(1)})^{T},\dots ,(\bar{\xi }^{(M)})^{T})^{T}
.
\end{equation}
Then $ \bar{\bxi }_{i-}=(\bGamma _{i-})^{-1}\mathbf{p}_{i-}\bXi _{i-}
$, where $ \bXi _{i-}=(\xi _{1},\dots ,\xi _{i-1},0,\xi _{i+1},\dots ,\xi
_{n})^{T} $. (Zero at the $i$-th row is a placeholder as in the matrix
$\mathbf{p}_{i-}$.) Applying (\ref{Eq9a}) one obtains
\begin{equation*}
\bar{\bxi }_{i-}=\bGamma ^{-1}(\mathbf{p}_{i-})^{T}\bXi _{i-}+
\frac{1}{1-h
_{i}}\bGamma ^{-1}\mathbf{p}_{i}^{\centerdot }(\mathbf{p}_{i}^{\centerdot
})^{T}\bGamma ^{-1}(\mathbf{p}_{i}^{\centerdot })^{T}\bXi _{i-}.
\end{equation*}
This together with $(\mathbf{p}_{i-})^{T}\bXi _{i-}=\mathbf{p}^{T}\bXi -
\mathbf{p}_{i}^{\centerdot }\xi _{i}^{T}$ implies
%
\begin{equation}
\label{Eq10}
\bar{\bxi }_{i-}=\bar{\bxi }+\frac{1}{1- h_{i}}\mathbf{a}_{i}^{\centerdot
}((\mathbf{p}_{i}^{\centerdot })^{T}\bar{\bxi }-\xi _{i}^{T}).
\end{equation}
Then the following algorithm allows one to calculate $
\hat{\mathbf{V}}^{(m)}$ for all $m=1,\dots ,M$ at once by $\sim Cn$
operations.

\begin{algorithm}[h!t]
\caption{\mbox{}}
\begin{algorithmic}
\STATE 1. Calculate $\bGamma $ and $\mathbf{a}$ by (\ref{Eq3}).
\STATE 2. Calculate $h_{i}$, $i=1,\dots ,n$ by (\ref{Eq9b}).
\STATE 3. Calculate $\bar{\bxi }$ by (\ref{Eq0}) and (\ref{Eqbarbxi}).
\STATE 4. Calculate $\bar{\bxi }_{i-}$ by (\ref{Eq10}).
\STATE 5. Calculate $\hat{\vartheta }_{i-}=\hat{\vartheta }_{i-}^{(m)}=H(\bar{\xi }_{i-1}^{(m)})$ for $i=1,\dots ,n$, $m=1,\dots ,M$.
\STATE 6. Calculate $\hat{\mathbf{V}}^{(m)}_{;n}$ by (\ref{Eq6a}) for all $m=1,\dots ,M$.
\end{algorithmic}
\end{algorithm}

%
%
%
%
%
%
%

\section{Regression with errors in variables}%
\label{SecErV}

In this section we consider a mixture of simple linear regressions with
errors in variables. A modification of orthogonal regression estimation
technique was proposed for the regression coefficients\index{regression coefficients} estimation in
\cite{MNS}. We will show how the jackknife ACM estimators\index{jackknife ! ACM estimators} from
Section~\ref{SecMain} can be applied in this case to construct
confidence sets for the regression coefficients.\index{regression coefficients}

Recall the errors-in-variables regression model in the context of
mixture with varying concentrations.

We consider the case when each subject $O$ has two variables of
interest: $x(O)$ and $y(O)$. These variables are related by a strict
linear dependence with coefficients depending on the component that
$O$ belongs to:
%
\begin{equation}
\label{EqO1}
y(O)=b_{0}^{(\kappa (O))}+b_{1}^{(\kappa (O))}x(O),
\end{equation}
where $b_{0}^{(m)}$, $b_{1}^{(m)}$ are the regression coefficients\index{regression coefficients} for
the $m$-th component.\vadjust{\goodbreak}

The true values of $x(O)$ and $y(O)$ are unobservable. These variables
are observed with measurement errors
%
\begin{equation}
\label{EqO2}
X(O) =x(O)+\varepsilon _{X}(O),\qquad  Y(O)=y(O)+\varepsilon _{Y}(O).
\end{equation}
Here we assume that the errors $\varepsilon _{X}(O)$ and $\varepsilon
_{Y}(O)$ are conditionally independent given $\kappa (O)=m$,
%
\begin{equation}
\label{EqO3}
\Mvtex^{(m)}\varepsilon _{X}=\Mvtex^{(m)}\varepsilon _{Y}=0
\quad \text{ and }\quad  \Dvtex^{(m)}\varepsilon _{X}=\Dvtex^{(m)}\varepsilon _{Y}=
\sigma ^{2}_{(m)}
\end{equation}
for all $m=1,\dots ,M$. So the distributions of $\varepsilon _{X}(O)$ and
$\varepsilon _{Y}(O)$ can be different, but their variances are the same
for a given subject. We assume that $\sigma ^{2}_{(m)}>0$, $m=1,\dots
,M$, and are unknown.

As in Section~\ref{SecMixt} we observe a sample $(X(O_{j}),Y(O_{j}))^{T}=(X
_{j},Y_{j})^{T}$, $j=1,\dots ,n$, from the mixture with known
concentrations $p_{j}^{(m)}=\prvtex \{\kappa (O_{j})=m\}$.

In the case of homogeneous sample, when there is no mixture, the
classical way to estimate $b_{0}$ and $b_{1}$ is orthogonal regression.
That is, the estimator is taken as the minimizer of the total least squares
functional which is the sum of squares of minimal Euclidean distances
from the observation points to the regression line. The modification of
this technique for mixtures with varying concentrations proposed in
\cite{MNS} leads to the following estimators for $b_{0}^{(k)}$ and
$b_{1}^{(k)}$:
%
\begin{equation}
\label{EqO4}
\begin{aligned}
\hat{b}_{1}^{(k)}
&=
\frac{
\hat{S}_{XX}^{(k)}-\hat{S}_{YY}^{(k)} +\sqrt{(
\hat{S}_{XX}^{(k)}-\hat{S}_{YY}^{(k)})^{2}+4(\hat{S}_{XY}^{(k)})^{2}}
}{2\hat{S}_{XY}^{(k)}
},
\\
\hat{b}_{0}^{(k)}&=\bar{Y}^{k} -\hat{b}
_{1}^{(k)}\bar{X}^{(k)},
\end{aligned}
\end{equation}
where
\begin{gather*}
\bar{X}^{(k)}=\sum _{j=1}^{n} a_{j}^{(k)}X_{j},\qquad  \bar{Y}^{(k)}=\sum
_{j=1}^{n} a_{j}^{(k)}Y_{j},
\\
\hat{S}_{XX}^{(k)}=\sum _{j=1}^{n} a_{j}^{(k)}(X_{j}-\bar{X}^{(k)})^{2},
\qquad\hat{S}_{YY}^{(k)}=\sum _{j=1}^{n}
a_{j}^{(k)}(Y_{j}-\bar{Y}^{(k)})^{2},
\\
\hat{S}_{XY}^{(k)}=\sum _{j=1}^{n} a_{j}^{(k)}(X_{j}-\bar{X}^{(k)})(Y
_{j}-\bar{Y}^{(k)}).
\end{gather*}
Conditions of consistency and asymptotic normality of these estimators are
given in \xch{Theorems}{theorems} 5.1 and 5.2 from \cite{MNS}. For example, under the
assumptions of Theorem~\ref{Th4} we obtain
\begin{equation*}
\sqrt{n}(\hat{\vartheta }_{;n}^{(k)}-\vartheta ^{(k)})\to N(0,
\mathbf{V}_{\infty }^{(k)}),
\end{equation*}
where $\vartheta ^{(k)}=(b_{0}^{(k)},b_{1}^{(k)})^{T}$.

The ACM $\mathbf{V}_{\infty }$ of the estimator is given by formula (21)
in \cite{MNS}. This formula is rather complicated and involves
theoretical moments of unobservable variables $x(O)$, $\varepsilon
_{X}(O)$ and $\varepsilon _{Y}(O)$. So it is natural to estimate
$\mathbf{V}_{\infty }$ by the jackknife technique,\index{jackknife ! technique} which doesn't need
to know or estimate these moments.

Notice that the estimator $(\hat{b}_{0}^{(k)},\hat{b}_{1}^{(k)})^{T}$
can be represented in terms of Section~\ref{SecMain} if we expand the
space of observable variables including quadratic terms. That is, we
consider the sample
\begin{equation*}
\xi _{j}=(X_{j},Y_{j},(X_{j})^{2},(Y_{j})^{2},(X_{j}Y_{j})^{2})^{T},\quad  j=1,
\dots ,n.
\end{equation*}
Then the estimator $\hat{\vartheta }_{;n}=(\hat{\vartheta }^{1}_{;n},
\hat{\vartheta }^{2}_{;n})^{T}=(\hat{b}_{0}^{(k)},\hat{b}_{1}^{(k)})^{T}$
defined by (\ref{EqO4}) can be represented in form (\ref{Eq5}) with
twice continuously differentiable function $H$ if $\Dvtex^{(k)}[x]
\neq0$ and $b_{1}^{(k)}\neq0$.

So we can apply the technique developed in
Sections~\ref{SecMain}--\ref{SecAlg}. Let us define the estimator
$\hat{\mathbf{V}}_{;n}^{(k)}$ for $\mathbf{V}^{(k)}_{\infty }$ by
(\ref{Eq6a}).
%
\begin{thm}
\label{Th4}
Assume that the following conditions hold.
\begin{enumerate}
\item $\frac{1}{n}\bGamma _{;n}\to \bGamma _{\infty }$ as $n\to \infty $
and $\det \bGamma _{\infty }>0$.
\item  Assumption~(\ref{EqLim}) holds.
\item  $\Mvtex^{(m)}[(x)^{12}]\,{<}\,\infty $, $\Mvtex^{(m)}[(\varepsilon _{X})^{12}]\,{<}\,
\infty $, $\Mvtex^{(m)}[(\varepsilon _{Y})^{12}]\,{<}\,\infty $ for all
$m=1,\dots ,M$.
\item $\Dvtex^{(k)}x(0)\neq0$ and $b_{1}^{(k)}\neq0$.
\end{enumerate}

Then $\hat{\mathbf{V}}_{;n}^{(k)}\to \mathbf{V}_{\infty }^{(k)}$ in
probability as $n\to \infty $.
\end{thm}
This theorem is a simple combination of Theorem~\ref{Th3} and
\xch{Theorem}{theorem}~5.2 from \cite{MNS}.

In what follows we assume that $\mathbf{V}_{\infty }^{(k)}$ is
nonsingular. This assumption holds, e.g. if for all $m$ the
distributions of $x(O)$, $\varepsilon _{X}(O)$ and $\varepsilon _{Y}(O)$
given $\kappa (O)=m$ are absolutely continuous with continuous PDFs.
(The proof of this fact is rather technical, so we do not present it
here.)

We can construct a confidence set (ellipsoid) for the unknown parameter
$\vartheta ^{(k)}$ applying the Theorem~\ref{Th4} by the usual way.
Namely, let for any $\mathbf{t}\in \mathbb{R}^{2}$
\begin{equation*}
T_{;n}(\mathbf{t})=(\mathbf{t}-\hat{\vartheta }^{(k)}_{;n})^{T}(
\hat{\mathbf{V}}_{;n}^{(k)})^{-1}(\mathbf{t}-\hat{\vartheta }^{(k)}
_{;n}).
\end{equation*}
Then in the assumptions of Theorem~\ref{Th4}, if $\det \mathbf{V}_{
\infty }^{(k)}\neq0$,
%
\begin{equation}
\label{EqAsChi}
T_{;n}(\vartheta ^{(k)})\weak \eta ,\quad  \text{ as } n\to \infty ,
\end{equation}
where $\eta $ is a random variable (r.v.) with chi-square distribution
with 2 degrees of freedom.

Consider $B_{\alpha ;n}=\{\mathbf{t}\in \mathbb{R}^{2}: T_{;n}(
\mathbf{t})\le Q^{\eta }(1-\alpha ) \}$, where $Q^{\eta }(\alpha )$
means the quantile of level $\alpha $ for the r.v. $\eta $. By
(\ref{EqAsChi})
%
\begin{equation}
\label{EqConfEll}
\prvtex \{\vartheta ^{(k)}\in B_{\alpha ;n} \}\to 1-\alpha ,
\end{equation}
so $B_{\alpha ;n}$ is an asymptotic confidence set for $\vartheta ^{(k)}$
of level $\alpha $.

\section{Results of simulation}%
\label{SecSimul}

To assess performance of the proposed technique we performed a small
simulation study. In the following three experiments we calculated
covering frequencies of confidence sets for regression coefficients\index{confidence sets for regression coefficients} in
the model (\ref{EqO1})--(\ref{EqO3}) constructed by (\ref{EqConfEll}) and
corresponding one-dimensional confidence intervals.\index{confidence intervals}

In all experiments for sample size $n=100$ through $5000$ we generated
$B=1000$ samples and calculated estimates for the parameters and
corresponding confidence sets.\index{confidence sets} The one-dimensional confidence intervals\index{confidence intervals}
for $b_{i}^{(k)}$ were calculated by the standard formula
\begin{equation*}
\left [\hat{b}_{i;n}^{(k)} -\lambda _{\alpha /2}\sqrt{\frac{\hat{v}
^{(k)}_{ii;n}}{n}},\hat{b}_{i;n}^{(k)} +\lambda _{\alpha /2}\sqrt{\frac{
\hat{v}^{(k)}_{ii;n}}{n}}\ \right ],
\end{equation*}
where $\hat{v}^{(k)}_{ii;n}$ is the $i$-th diagonal entry of the matrix
$\hat{\mathbf{V}}_{;n}^{(k)}$, $\lambda _{\alpha /2}$ is the quantile of
level $1-\alpha /2$ for the standard normal distribution. The confidence
level for the sets and intervals was taken $\alpha =0.05$.

Then the numbers of cases when the confidence set covers the true value
of the estimated parameter were calculated and divided by $B$. These are
the covering frequencies reported in the tables below.

In all the experiments we considered two-component mixture ($M=2$) with
the concentrations of components
\begin{equation*}
p_{j;n}^{1}=j/n,\qquad  p_{j;n}^{2}=1-j/n.
\end{equation*}
The regression coefficients\index{regression coefficients} were taken as
\begin{equation*}
b_{0}^{(1)}=1/2, \qquad  b_{1}^{(1)}=2,\qquad  b_{0}^{(2)}=-1/2,\qquad  b_{1}^{(2)}=-1/3,
\end{equation*}
and the distribution of the true (unobservable) regressor $x(O)$ was
$N(0,2)$ for $\kappa (O)=1$ and $N(1,2)$ for $\kappa (O)=2$.

\begin{experiment}\label{experiment1}
In this experiment we let $\varepsilon _{X}$ and
$\varepsilon _{Y}\sim N(0,0.25)$. The variance of the errors is so small
that the regression coefficients\index{regression coefficients} can be estimated with no difficulties
even for small sample sizes.

The covering frequencies for confidence sets\index{confidence sets} are presented in
Table~\ref{Tab1}. It seems that they approach the nominal covering
probability $0.95$ with satisfactory accuracy for sample sizes
$n\ge 1000$.
\end{experiment}

\begin{table}[ht]
\caption{Covering frequencies for confidence sets in Experiment 1}
\label{Tab1}
\begin{tabular}{rrrcrrc}
\hline
$n$ & $ b_{0}^{(1)}$ & $b_{1}^{(1)}$ & $(b_{0}^{(1)},b_{1}^{(1)})$ & $b_{0}^{(2)}$ & $b_{1}^{(2)}$ & $(b_{0}^{(2)},b_{1}^{(2)})$ \\
\hline
100 & 0.935 & 0.961 & 0.948 & 0.936 & 0.987 & 0.957 \\
250 & 0.953 & 0.960 & 0.950 & 0.964 & 0.980 & 0.950 \\
500 & 0.940 & 0.954 & 0.939 & 0.958 & 0.973 & 0.962 \\
1000 & 0.946 & 0.949 & 0.943 & 0.954 & 0.971 & 0.935 \\
2500 & 0.961 & 0.949 & 0.948 & 0.937 & 0.953 & 0.947 \\
5000 & 0.947 & 0.949 & 0.948 & 0.954 & 0.956 & 0.958 \\
\hline
\end{tabular}
\end{table}

\begin{experiment}\label{experiment2}
In this experiment we enlarged the variance of
the error terms taking it as $\sigma ^{2}=2$. All other parameters were
the same as in Experiment \ref{experiment1}. The results are presented in
Table~\ref{Tab2}.
\end{experiment}
%
\begin{table}[ht]
\caption{Covering frequencies for confidence sets in Experiment 2}
\label{Tab2}
\begin{tabular}{rrrcrrc}
\hline
$n$ & $b_{0}^{(1)}$ & $b_{1}^{(1)}$ & $(b_{0}^{(1)},b_{1}^{(1)})$ & $b_{0}^{(2)}$ & $b_{1}^{(2)}$ & $(b_{0}^{(2)},b_{1}^{(2)})$ \\
\hline
100 & 0.969 & 0.942 & 0.918 & 0.950 & 0.974 & 0.958 \\
250 & 0.958 & 0.956 & 0.945 & 0.946 & 0.962 & 0.959 \\
500 & 0.949 & 0.945 & 0.936 & 0.953 & 0.966 & 0.960 \\
1000 & 0.959 & 0.946 & 0.954 & 0.947 & 0.958 & 0.942 \\
2500 & 0.956 & 0.949 & 0.950 & 0.947 & 0.961 & 0.958 \\
5000 & 0.953 & 0.941 & 0.952 & 0.955 & 0.955 & 0.968 \\
\hline
\end{tabular}
\end{table}
It seems that the increase of errors dispersion doesn't deteriorate
covering accuracy of the confidence sets.\index{confidence sets}

\begin{experiment}\label{experiment3}
Here we consider the case when the errors
distributions are heavy tailed. We generate the data with $\varepsilon
_{X}$ and $\varepsilon _{Y}$ having Student-T distribution with \mbox{$\mathrm{df}\,{=}\,14$}
degrees of freedom. (This is the smallest df for which assumptions of
Theorem~\ref{Th4} hold.) Covering frequencies are presented in
\xch{Table~\ref{Tab3}.}{Table~\ref{Tab3}}
\end{experiment}

\begin{table}[ht]
\caption{Covering frequencies for confidence sets in Experiment 3}
\label{Tab3}
\begin{tabular}{rrrcrrc}
\hline
$n$ & $b_{0}^{(1)}$ & $b_{1}^{(1)}$ & $(b_{0}^{(1)},b_{1}^{(1)})$ & $b_{0}^{(2)}$ & $b_{1}^{(2)}$ & $(b_{0}^{(2)},b_{1}^{(2)})$ \\
\hline
100 & 0.935 & 0.961 & 0.948 & 0.936 & 0.987 & 0.957 \\
250 & 0.953 & 0.960 & 0.950 & 0.964 & 0.980 & 0.950 \\
500 & 0.940 & 0.954 & 0.939 & 0.958 & 0.973 & 0.962 \\
1000 & 0.946 & 0.949 & 0.943 & 0.954 & 0.971 & 0.935 \\
2500 & 0.961 & 0.949 & 0.948 & 0.937 & 0.953 & 0.947 \\
5000 & 0.947 & 0.949 & 0.948 & 0.954 & 0.956 & 0.958 \\
\hline
\end{tabular}
\end{table}
It seems that the accuracy of covering slightly decreased but this
decrease is insignificant for practical purposes.

\section{Sociologic analysis of EIT data}%
\label{SecEit}

We would like to demonstrate advantages of the proposed technique by
application to the analysis of the External Independent Testing\index{External Independent Testing (EIT)} (EIT)
data (\xch{see}{See} \cite{MM}). EIT is a set of exams for high school
graduates in Ukraine which must be passed for admission to universities.
We use data on EIT-2016 from the official site of \textit{Ukrainian
Center for Educational Quality Assessment}.\footnote{\url{https://zno.testportal.com.ua/stat/2016}}

In this presentation we consider only the data on scores on two
subjects: \textit{Ukrainian language and literature} (Ukr) and on
\textit{Mathematics} (Math). The scores range from 100 to 200 points.
(We have excluded the data on persons who failed on one of the exams or
didn't pass these exams at all.) EIT-2016 contain such data on 246
thousands of examinees. The information on the region (Oblast) of
Ukraine in which each examinee attended the high school is also
available in EIT-2016.

Our aim is to investigate how dependence between Ukr and Math scores
differs for examinees grown up in different environments There can be,
e.g. an environment of adherents of Ukrainian culture and Ukrainian
state, or in the environment of persons critical toward the Ukrainan
independence. EIT-2016 doesn't contain information on such issues. So
we use data on Ukrainian Parliament (Verhovna Rada) election results to
deduce approximate proportions of adherents of different political
choices in different regions of Ukraine.

We divided adherents of 29 parties and blocks that took part in the
elections into three large groups, which are the components of our
mixture:
\begin{enumerate}
\item[(1)] Pro-Ukrainian persons, voting for the parties that then created the
ruling coalition (BPP, Batkivschyna, Narodny Front, Radicals and
Samopomich)
\item[(2)] Contra-Ukrainian persons who voted for the Opposition block, voted
against all or voted for small parties which where under 5\% threshold
on these elections.
\item[(3)] Neutral persons who did not took part in the voting.
\end{enumerate}

Combining these data with EIT-2016 we obtain the sample $(X_{j},Y_{j})$,
$j\,{=}\,1,\dots ,n$, where $X_{j}$ is the Math score of the $j$-th examinee,
$Y_{j}$ is his/her Ukr score. The concentrations of components
$(p_{j}^{1},p_{j}^{2},p_{j}^{3})$ are taken as frequencies of adherents
of corresponding political choice at the region where the $j$-th
examinee attended the high school.

In \cite{MM} the authors propose to use classical linear
regression model (in which the error appears in the response only) to
describe dependence between $X_{j}$ and $Y_{j}$ in these data. But the
errors-in-variables model can be more adequate since the causes which
deteriorate ideal functional dependence $\mathrm{Ukr} = b_{0}+b_{1}$ Math can affect
both Math and Ukr scores causing random deviations of, maybe, the same
dispersion for each variable.

So, in this presentation, we assumed that the data are described by the
model (\ref{EqO1})--(\ref{EqO3}), where $\kappa (O)=1,2,3$ means the
component (environment at which the person $O$ was grown up)
corresponding to one of three political choices given above. (Lower and
upper endpoints of confidence intervals\index{confidence intervals} are given in columns named low
and upp correspondingly.)

\begin{table}[b!]
\caption{Confidence sets for coefficients of regression between Math and Ukr}
\label{TabR}
\begin{tabular}{c|cc|cc|cc}
\hline
&\multicolumn{2}{c|}{ Pro } &\multicolumn{2}{c|}{ Contra }&\multicolumn{2}{c}{ Neutral }\\
\hline
& low & upp & low & upp & low & upp \\
\hline
$b_{0}^{(k)}$ & 40.12 & 40.22 & 236.3 & 240.1 & 84.21 & 87.19 \\
$b_{1}^{(k)}$ & 0.8562 & -0.366 & -0.345 & -2.80 & 0.335 & 0.359 \\
\hline
\end{tabular}
\end{table}

In this model we calculated the confidence intervals\index{confidence intervals} of the level
$\alpha =0.05/3\approx 0.0167$ to derive the unilateral level
$\alpha =0.05$ in comparisons of three intervals derived for three
different components. The results are presented in Table~\ref{TabR}.
We observe that the obtained\vadjust{\goodbreak} intervals are rather narrow. They don't
intersect for different components. So, the regression coefficients\index{regression coefficients} for
different components are significantly different. (Of course, it is so
only if our theoretical model of the data distribution is adequate.)

The orthogonal regression lines corresponding to different components
are presented on Fig~\ref{Fig1}. The solid line corresponds to the
Pro-component, the dashed line for the Contra-component and the dotted
line for the Neutrals.
%
\begin{figure}
\includegraphics{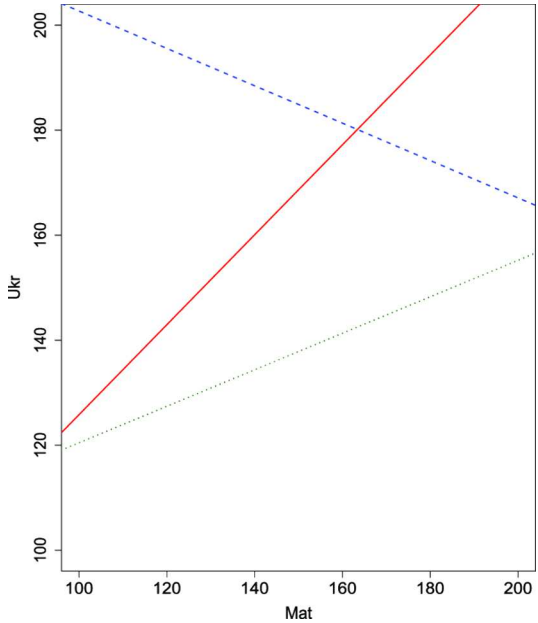}
\caption{Estimated orthogonal regression lines for EIT-2016 data}
\label{Fig1}
\end{figure}

These results have simple and plausible explanation. Say, in the Pro-component the success in Ukr positively correlates with the general
school successes, so with Math scores, too. It is natural for persons who
are interested in Ukrainian culture and literature. In the
Contra-component the correlation is negative. Why? The persons with high
Math grades in this component do not feel the need to learn Ukrainian.
But the persons with less success in Math try to improve their average
score (by which the admission to universities is made) by increasing
their Ukr score. The Neutral component shows positive correlation
between Math and Ukr, but it is less then the correlation in the Pro-component.

Surely, these explanations are too simple to be absolutely correct. We
consider them only as examples of hypotheses which can be deduced from
the data by the proposed technique.

\section{Proofs}%
\label{SecProof}

To demonstrate Theorem \ref{Th3} we need three lemmas. Below the symbols
$C$ and $c$ mean finite constants, maybe different.

\begin{lem}
\label{Lem1}
Assume that $\det \bGamma _{\infty }>0$. Then:
\begin{enumerate}
\item $\sup _{j=1,\dots ,n; m=1,\dots ,M}|a_{j;n}^{m}|=O(n^{-1})$.
\item $\sup _{i,j=1,\dots ,n; i\neq j; m=1,\dots ,M}|a_{j;n}^{m}-a_{ji-;n}
^{m}|=O(n^{-2}\xch{).}{)}$
\end{enumerate}
\end{lem}
\begin{proof}
By definition, $\frac{1}{n}\bGamma _{;n}\to \bGamma
_{\infty }$, so there exists $c>0$ such that $\det \bGamma _{;n}>c$ for
all $n$ large enough. This together with $|p_{j;n}^{m}|<1$ imply
%
\begin{equation}
\label{Eq11}
\|\bGamma _{;n}^{-1}\|\le \frac{C}{n}.
\end{equation}
(Here $\|\cdot \|$ means the operator norm.) Taking into account that
$\mathbf{a}_{;n}=\mathbf{p}_{;n}\bGamma _{;n}^{-1}$, we obtain the first
statement of the lemma.

Then by (\ref{Eq9a})--(\ref{Eq9b}),
\begin{equation*}
\mathbf{a}_{j}^{\centerdot }-\mathbf{a}_{ji-}^{\centerdot }=
\bGamma
^{-1}\mathbf{p}_{j}^{\centerdot }-\bGamma ^{-1}_{i-}\mathbf{p}_{j}
^{\centerdot }
=
\frac{1}{h_{i}}\bGamma ^{-1}\mathbf{p}_{i}^{\centerdot
}(\mathbf{p}_{i}^{\centerdot })^{T}\bGamma ^{-1}\mathbf{p}_{j}^{
\centerdot }.
\end{equation*}
This together with (\ref{Eq11}) yields the second statement.
\end{proof}

\begin{lem}
\label{Lem2}
Assume that for $m=1,\dots ,M$:
\begin{enumerate}
\item $\Mvtex^{(m)}[\xi ]=0$;
\item  For some $\delta >0$, $\Mvtex^{(m)}(\xi )^{2}|\log |\xi ||^{(1+
\delta )}<\infty $;
\item $\det \bGamma _{\infty }>0$.
\end{enumerate}

Then for some $C<\infty $,
\begin{equation*}
\prvtex \left \{
|\bar{\xi }^{(k)}|>C\sqrt{\frac{\log \log n}{n}}\right \}  \to 0\quad \text{
as } n\to \infty .
\end{equation*}
\end{lem}
\begin{proof}
Let $\eta _{1}$, \dots , $\eta _{n}$ be independent random
variables with $\Mvtex \eta _{i}=0$. Let us denote $B_{n}=\sum _{j=1}
^{n} \Mvtex (\eta _{j})^{2}$. Then the last formula in the proof of \xch{Theorem}{theorem}
7.2 and \xch{Theorem}{theorem} 7.3 in \cite{Petrov} imply the following
\xch{proposition}{Proposition}.
%
\begin{prop}
\label{Prop1}
If
\begin{equation*}
\varlimsup _{n\to \infty }\frac{1}{n}\sum _{j=1}^{n}\Mvtex \eta _{j}^{2}|
\log |\eta _{j}||^{1+\delta }<\infty ,
\end{equation*}
then for any $b$ such that $0<b<\sqrt{1+\delta }$,
\begin{equation*}
\prvtex \left \{
\sum _{j=1}^{n}\eta _{j}\ge b\sqrt{2B_{n}\log
\log B_{n}}\right \}
\le (\log B_{n})^{-b^{2}}
\end{equation*}
for $n$ large enough.
\end{prop}
Let $\eta _{j}=\pm na_{j;n}^{(k)}\xi _{j}^{l}$. Then $ B_{n}=n^{2}\sum
_{j=1}^{n}(a_{j;n}^{(k)})^{2}\Dvtex \xi _{j}^{l}\sim Cn $ by
Lemma~\ref{Lem1}. Assumption 2 implies that the assumption of
Proposition~\ref{Prop1} holds. So,
\begin{equation*}
\prvtex \left \{  \sum _{j=1}^{n} na_{j;n}^{(k)}\xi _{j}^{l}>b\sqrt{2Cn
\log \log Cn } \right \}  \to 0.
\end{equation*}
This implies the statement of the \xch{lemma}{Lemma}.
\end{proof}

\begin{lem}
\label{Lem3}
Assume that for some $\alpha >0$
\begin{equation*}
\Mvtex^{(m)} [|\xi |^{\alpha }]<\infty\quad  \text{ for all } m=1,\dots , M.
\end{equation*}
Then for any $\beta >1/\alpha $ there exists $C<\infty $ such that
\begin{equation*}
\prvtex \{\sup _{j=1,\dots ,n}|\xi _{j}|>Cn^{\beta } \}\to 0\quad
\text{ as } n\to \infty .
\end{equation*}
\end{lem}
\begin{proof}
By the Chebyshov inequality we obtain that for some
$0<R<\infty $,
\begin{equation*}
\prvtex \{|\xi _{j}|>x \}\le \frac{R}{x^{\alpha }}.
\end{equation*}
Then for $\alpha \beta >1$,
\begin{gather*}
\prvtex \{\sup _{j=1,\dots ,n}|\xi _{j}|>cn^{\beta }\}=
1-\prvtex \{
\sup _{j=1,\dots ,n}|\xi _{j}|\le cn^{\beta }\}
\\
=1-\prod _{j=1}^{n}\prvtex \{|\xi _{j}|\le Cn^{\beta }\}=1-\prod _{j=1}
^{n}(1-\prvtex \{|\xi _{j}|> Cn^{\beta }\})
\\
\le 1-\left (1-\frac{R}{C^{\alpha }n^{\alpha \beta }}\right )^{n}=1-
\exp \left (n\log \left (1-\frac{R}{C^{\alpha }n^{\alpha \beta }}
\right ) \right )
\\
\sim 1-\exp \left (\frac{-nR}{C^{\alpha }n^{\alpha \beta }} \right )
\to 0 \quad \text{ as } n\to \infty ,
\end{gather*}
if $\alpha \beta >1$.

Lemma is proved. 
\end{proof}

\begin{proof}[Proof of Theorem \ref{Th3}]
Let $\xi _{j}'=\xi _{j}-\Mvtex \xi _{j}$. Then
\begin{equation*}
\bar{\xi }^{(k)}=\sum _{j=1}^{n} a_{j}^{(k)}\xi _{j}=
\sum _{j=1}^{n} a
_{j}^{(k)}\xi '_{j}
+\sum _{j=1}^{n}\sum _{m=1}^{M} a_{j}^{(k)}p_{j}
^{(m)}\mu ^{(m)}
=\sum _{j=1}^{n} a_{j}^{(k)}\xi _{j}'+\mu ^{(k)},
\end{equation*}
due to (\ref{Eq1}). Similarly,
\begin{equation*}
\bar{\xi }^{(k)}_{i-}=\sum _{j\neq i}a_{ji-}^{(k)}\xi _{j}'+\mu ^{(k)}.
\end{equation*}
Let us denote $\mathbf{U}_{i}=(U_{1},\dots ,U_{d})^{T}=\bar{\xi }^{(k)}-\bar{
\xi }_{i-}^{(k)}$. Then
%
\begin{equation}
\label{Eq11a}
\mathbf{U}_{i}=\sum _{j=1}^{n} a_{j}^{(k)}\xi _{j}'-\sum _{j\neq i}a
_{ji-}\xi _{j}'
=a_{i}^{(k)}\xi _{i}'+\sum _{j\neq i}(a_{j}^{(k)}-a_{ji-}
^{(k)})\xi _{j}'
\end{equation}
and
\begin{equation*}
\hat{\vartheta }-\hat{\vartheta }_{i-}=H(\bar{\xi }^{(k)})-H(\bar{
\xi }^{(k)}_{i-})
=\mathbf{H}'(\zeta _{i})\mathbf{U}_{i},
\end{equation*}
where $\zeta _{i}$ is some intermediate point between $\bar{\xi }^{(k)}$
and $\bar{\xi }^{(k)}_{i-}$. So,
%
\begin{equation}
\label{Eq12}
\hat{\mathbf{V}}_{;n}=n\sum _{i=1}^{n} \mathbf{H}'(\zeta _{i})
\mathbf{U}_{i}\mathbf{U}_{i}^{T}(\mathbf{H}'(\zeta _{i}))^{T}.
\end{equation}
Let us denote
%
\begin{equation}
\label{Eq13}
\tilde{\mathbf{V}}_{;n}=n\sum _{i=1}^{n}\mathbf{H}'(\mu ^{(k)})
\mathbf{U}_{i}\mathbf{U}_{i}^{T}(\mathbf{H}'(\mu ^{(k)}))^{T}.
\end{equation}
We will show that
%
\begin{equation}
\label{Eq14}
\tilde{\mathbf{V}}_{;n}\to \mathbf{V}_{\infty }\quad \text{\textrm{ as }} n\to
\infty \text{ in probability}
\end{equation}
and
%
\begin{equation}
\label{Eq15}
n\|\hat{\mathbf{V}}_{;n}-\tilde{\mathbf{V}}_{;n}\|\to 0\quad  \text{\textrm{ as }}
n\to \infty \text{ in probability}.
\end{equation}
These two equations imply the statement of the theorem.

We start from (\ref{Eq14}). Let us calculate $\Mvtex
\tilde{\mathbf{V}}_{;n}$. Notice that
\begin{equation*}
\Mvtex \mathbf{U}_{i}\mathbf{U}_{i}^{T}=(a_{i}^{(k)})^{2}\Mvtex \xi
_{i}'(\xi _{i}')^{T}+
\sum _{j\neq i}(a_{j}^{(k)}-a_{ji-}^{(k)})^{2}
\Mvtex \xi _{i}'(\xi _{i}')^{T}.
\end{equation*}
By  \xch{Assumption}{assumption} 2 of the theorem, $ \sup _{i}\|\Mvtex \xi _{i}'(\xi
_{i}')^{T}\|<C $, and by Lemma \ref{Lem1},
\begin{equation*}
\sup _{j=1,\dots ,n}(a_{j}^{(k)}-a_{ji-}^{(k)})^{2}=O(n^{-4}).
\end{equation*}
So,
\begin{equation*}
\Mvtex \tilde{\mathbf{V}}_{;n}=n\mathbf{H}'(\mu ^{(k)})\sum _{i=1}^{n}(a
_{i}^{(k)})^{2}\Mvtex \xi _{i}'(\xi _{i}')^{T}(\mathbf{H}'(\mu ^{(k)}))^{T}+O(n
^{-1}).
\end{equation*}
By the same way as in (\ref{Eq5}), we obtain
%
\begin{equation}
\label{Eq16}
\Mvtex \tilde{\mathbf{V}}_{;n}\to \mathbf{H}'(\mu ^{(k)})\bSigma _{
\infty }^{(k)}(\mathbf{H}'(\mu ^{(k)}))^{T}=\mathbf{V}_{\infty }.
\end{equation}
Now, let us estimate
\begin{gather}
\Mvtex \|\tilde{\mathbf{V}}_{;n}-\Mvtex \tilde{\mathbf{V}}_{;n}\|^{2}
\le Cn^{2}\sum _{l_{1},l_{2}=1}^{d}\sum _{i=1}^{n}
\Mvtex (U_{i}^{l_{1}}U
_{i}^{l_{2}}-\Mvtex U_{i}^{l_{1}}U_{i}^{l_{2}})^{2}
\nonumber
\\
\label{Eq17}
\le Cn^{2}\sum _{l=1}^{n}\sum _{i=1}^{n}\Mvtex (U_{i}^{l})^{4}.
\end{gather}
Notice that
\begin{gather*}
\Mvtex (U_{i}^{l})^{4}=\Mvtex \left(a_{i}^{(k)}{\xi _{i}^{l}}'+\sum _{j
\neq i}(a_{i}^{(k)}-a_{ji-}^{(k)}){\xi _{j}^{l}}'\right)^{4}
\\
=(a_{i}^{(k)})^{4}\Mvtex ({\xi _{j}^{l}}')^{4}+
6 (a_{i}^{(k)})^{2}
\Mvtex \left (\sum _{j\neq i}(a_{i}^{(k)}-a_{ji-}^{(k)}){\xi _{j}^{l}}'\right )
^{2}
\\
+\Mvtex \left (\sum _{j\neq i}(a_{i}^{(k)}-a_{ji-}^{(k)}){\xi _{j}^{l}}'\right )
^{4}=O(n^{-4})
\end{gather*}
due to Lemma~\ref{Lem1} and Assumption 2 of the theorem. So by
(\ref{Eq17}) we obtain
\begin{equation*}
\Mvtex \|\tilde{\mathbf{V}}_{;n}-\Mvtex \tilde{\mathbf{V}}_{;n}\|^{2}=O(n
^{-1}).
\end{equation*}
This and (\ref{Eq16}) imply (\ref{Eq14}).

Let us show (\ref{Eq15}). Notice that
\begin{gather}
n(\hat{\mathbf{V}}_{;n}-\tilde{\mathbf{V}}_{;n})=
n\sum _{i=1}^{n}(
\mathbf{H}'(\zeta _{i})-\mathbf{H}'(\mu ^{(k)}))\mathbf{U}_{i}(
\mathbf{U}_{i})^{T}(\mathbf{H}'(\zeta _{i}))^{T}
\nonumber
\\
\label{Eq18}
+n\sum _{i=1}^{n}(\mathbf{H}'(\mu ^{(k)}))\mathbf{U}_{i}(\mathbf{U}_{i})^{T}(
\mathbf{H}'(\zeta _{i})-\mathbf{H}'(\mu ^{(k)}))^{T}.
\end{gather}
By Lemma~\ref{Lem1}, $ \sup _{i}\|\mathbf{U}_{i}(\mathbf{U}_{i})^{T}\|
\le Cn^{-2}\sup _{i}\|\xi _{i}'\|^{2} $. By Lemma~\ref{Lem3} and
\xch{Assumption}{assumption}~2 of the theorem,
\begin{equation*}
\sup _{i}\|\xi _{i}'\|^{2}=O_{P}(n^{\beta })\quad  \text{ for } \beta >\frac{2}{
\alpha }.
\end{equation*}
Since $\alpha >2$ we may take here $\beta <1/2$. Let us estimate $
\sup _{i}\|\mathbf{H}'(\zeta _{i})-\mathbf{H}'(\mu ^{(k)})\| $. Notice that
$\zeta _{i}$ is an intermediate point between $\bar{\zeta }^{(k)}$ and
$\bar{\zeta }_{i-}^{(k)}$. By Lemma~\ref{Lem2},
\begin{equation*}
\|\bar{\xi }^{(k)}-\mu ^{(k)}\|=O_{P}\left (\sqrt{
\frac{\log \log n}{n}} \right ).
\end{equation*}
Then, by Lemma~\ref{Lem1},
\begin{equation*}
\sup _{i}\|\bar{\xi }^{(k)}-\bar{\xi }^{(k)}_{i-}\|\le Cn^{-1}\sup _{i}|
\xi _{i}'|=O_{P}(n^{-1+\beta })
\end{equation*}
and
\begin{equation*}
\sup _{i}\|\zeta _{i}-\mu ^{(k)}\|=O_{P}\left (\sqrt{
\frac{\log \log n}{n}} \right ).
\end{equation*}
Due to  \xch{Assumption}{assumption}~1 of the theorem this implies
\begin{equation*}
\sup _{i}\|\mathbf{H}'(\zeta _{i})-\mathbf{H}'(\mu ^{(k)})\|=O_{P}\left (\sqrt{\frac{
\log \log n}{n}} \right )
\end{equation*}
and $\sup _{i}\|\mathbf{H}'(\zeta _{i})|=O_{P}(1)$.\newpage

Combining these estimates with (\ref{Eq18}), we obtain
\begin{equation*}
n(\hat{\mathbf{V}}_{;n}-\tilde{\mathbf{V}}_{;n})=
n\sum _{i=1}^{n} O
_{P}\left (\sqrt{\frac{\log \log n}{n}} \right )\frac{C}{n^{2}}n
^{\beta }O_{P}(1)=O_{P}(1),
\end{equation*}
since $\beta <1/2$. This is (\ref{Eq15}).

Combining (\ref{Eq15}) and (\ref{Eq14}), we obtain the statement of
\xch{theorem}{Theorem}.
\end{proof}

\section{Conclusions}%
\label{SecConcl}

We introduced a modification of the jackknife technique\index{jackknife ! technique} for the ACM
estimation for moment estimators by observations from mixtures with
varying concentrations. A fast algorithm is proposed which implements
this technique. Consistency of derived estimator is demonstrated.
Results of simulations demonstrate its practical applicability for
sample sizes $n>1000$.
%





\begin{thebibliography}{99}

\bibitem{Borovkov}
\begin{bbook}
\bauthor{\bsnm{Borovkov}, \binits{A.A.}}:
\bbtitle{Mathematical Statistics}.
\bpublisher{Gordon and Breach Science Publishers},
\blocation{Amsterdam}
(\byear{1998}).
\bid{mr={1712750}}
\end{bbook}
%
\OrigBibText
Borovkov A.A.: Mathematical Statistics. Gordon and Breach Science
Publishers, Amsterdam (1998)
\endOrigBibText
\bptok{structpyb}
\endbibitem

\bibitem{Branham}
\begin{bchapter}
\bauthor{\bsnm{Branham}, \binits{R.}}:
\bctitle{Total Least Squares in Astronomy}.
In: \bbtitle{Total Least Squares and Errors-in-Variables Modeling},
pp.~\bfpage{375}--\blpage{384}.
\bpublisher{Springer},
\blocation{Dordrecht}
(\byear{2002}).
\bid{doi={10.1007/978-94-017-3552-0\_33}, mr={1952962}}
\end{bchapter}
%
\OrigBibText
Branham R.: Total Least Squares in Astronomy. In Total Least Squares and
Errors-in-Variables Modeling. Springer, Dordrecht, 375--384 (2002)
\endOrigBibText
\bptok{structpyb}
\endbibitem

\bibitem{Cheng}
\begin{bbook}
\bauthor{\bsnm{Cheng}, \binits{C.-L.}},
\bauthor{\bparticle{Van} \bsnm{Ness}, \binits{J.}}:
\bbtitle{Statistical Regression with Measurement Error, Kendall's Library of Statistics 6}.
\bpublisher{Arnold},
\blocation{London}
(\byear{1999}).
\bid{mr={1719513}}
\end{bbook}
%
\OrigBibText
Cheng C.-L., Van Ness J.: Statistical Regression with Measurement Error,
Kendall's Library of Statistics 6. Arnold, London (1999)
\endOrigBibText
\bptok{structpyb}
\endbibitem

\bibitem{Maiboroda2003}
\begin{bbook}
\bauthor{\bsnm{Maiboroda}, \binits{R.}}:
\bbtitle{Statistical analysis of mixtures}.
\bpublisher{Kyiv University Publishers},
\blocation{Kyiv}
\bcomment{(in Ukrainian)}
(\byear{2003})
\end{bbook}
%
\OrigBibText
Maiboroda, R.: Statistical analysis of mixtures. Kyiv University
Publishers, Kyiv (in Ukrainian) (2003)
\endOrigBibText
\bptok{structpyb}
\endbibitem

\bibitem{MS2012-2}
\begin{barticle}
\bauthor{\bsnm{Maiboroda}, \binits{R.}},
\bauthor{\bsnm{Sugakova}, \binits{O.}}:
\batitle{Statistics of mixtures with varying concentrations with application
to DNA microarray data analysis}.
\bjtitle{J. Nonparametr. Stat.}
\bvolume{24}(\bissue{1}),
\bfpage{201}--\blpage{215}
(\byear{2012}).
\bid{doi={10.1080/10485252.2011.630076}, mr={2885834}}
\end{barticle}
%
\OrigBibText
Maiboroda, R., Sugakova O.: Statistics of mixtures with varying
concentrations with application to DNA microarray data analysis.
Nonparametric statistics \textbf{24(1)}, 201--215 (2012)
\endOrigBibText
\bptok{structpyb}
\endbibitem

\bibitem{MNS}
\begin{barticle}
\bauthor{\bsnm{Maiboroda}, \binits{R.}},
\bauthor{\bsnm{Navara}, \binits{H.}},
\bauthor{\bsnm{Sugakova}, \binits{O.}}:
\batitle{Orthogonal regression for observations from mixtures}.
\bjtitle{Teor. Imovir. Mat. Stat.}
\bvolume{99},
\bfpage{152}--\blpage{167}
(\byear{2018})
\end{barticle}
%
\OrigBibText
Maiboroda R., Navara H., Sugakova O.: Orthogonal regression for
observations from mixtures. Teoriya Imovirnostei ta Matematychna
Statystyka \textbf{99}, 152-167 (2018)
\endOrigBibText
\bptok{structpyb}
\endbibitem

\bibitem{MM}
\begin{barticle}
\bauthor{\bsnm{Miroshnichenko V}, \binits{M.R.}}:
\batitle{Confidence ellipsoids for regression coefficients by observations from a mixture}.
\bjtitle{Mod. Stoch. Theory Appl.}
\bvolume{5}(\bissue{2}),
\bfpage{225}--\blpage{245}
(\byear{2018}).
\bid{doi={10.15559/18-VMSTA105}, doi={10.15559/18-vmsta105}, mr={3813093}}
\end{barticle}
%
\OrigBibText
Maiboroda R. Miroshnichenko V. Confidence ellipsoids for regression
coefficients by observations from a mixture. Modern Stochastics: Theory
and Applications, \textbf{5}, Iss.2 225--245 (2018)
\endOrigBibText
\bptok{structpyb}
\endbibitem

\bibitem{Kukush}
\begin{bbook}
\beditor{\bsnm{Masiuk}, \binits{S.}},
\beditor{\bsnm{Kukush}, \binits{A.}},
\beditor{\bsnm{Shklyar}, \binits{S.}},
\beditor{\bsnm{Chepurny}, \binits{M.}},
\beditor{\bsnm{Likhtarov}, \binits{I.}} (eds.):
\bbtitle{Radiation Risk Estimation: Based on Measurement Error Models},
\bedition{2}nd edn.
\bsertitle{de Gruyter series in Mathematics and Life Sciences},
vol.~\bseriesno{5}.
\bpublisher{de Gruyter}
(\byear{2017}).
\bid{mr={3726857}}
\end{bbook}
%
\OrigBibText
Masiuk S., Kukush A., Shklyar S., Chepurny M., Likhtarov I. (ed.):
Radiation Risk Estimation: Based on Measurement Error Models (2nd ed).
(de Gruyter series in Mathematics and Life Sciences, vol. 5). de
Gruyter, (2017)
\endOrigBibText
\bptok{structpyb}
\endbibitem

\bibitem{MLR}
\begin{barticle}
\bauthor{\bsnm{McLachlan}, \binits{G.J.}},
\bauthor{\bsnm{Lee}, \binits{S.X.}},
\bauthor{\bsnm{Rathnayake}, \binits{S.I.}}:
\batitle{Finite Mixture Models}.
\bjtitle{Ann. Rev. Stat. Appl.}
\bvolume{6},
\bfpage{355}--\blpage{378}
(\byear{2019}).
\bid{doi={10.1146/annurev-statistics-\\031017-100325}, mr={3939525}}
\end{barticle}
%
\OrigBibText
McLachlan G. J., Lee S.X., Rathnayake S. I.: Finite Mixture Models. Ann.
Rev. Stat. Appl. \textbf{6} 355-378 (2019)
\endOrigBibText
\bptok{structpyb}
\endbibitem

\bibitem{Petrov}
\begin{botherref}
\oauthor{\bsnm{Petrov}, \binits{V.}}:
Limit theorems of probability theory: sequences of independent random variables.
Clarendon Press
\end{botherref}
%
\OrigBibText
Petrov V.: Limit theorems of probability theory: sequences of
independent random variables. Clarendon Press
\endOrigBibText
\bptok{structpyb}
\endbibitem

\bibitem{Shao}
\begin{bbook}
\bauthor{\bsnm{Shao}, \binits{J.}}:
\bbtitle{Mathematical statistics}.
\bpublisher{Springer},
\blocation{New York}
(\byear{2007}).
\bid{doi={10.1007/b97553}, mr={2002723}}
\end{bbook}
%
\OrigBibText
Shao, J.: Mathematical statistics. Springer-Verlag: New York (2007)
\endOrigBibText
\bptok{structpyb}
\endbibitem

\bibitem{ShaoTu}
\begin{bbook}
\bauthor{\bsnm{Shao}, \binits{J.}},
\bauthor{\bsnm{Tu}, \binits{D.}}:
\bbtitle{The Jackknife and Bootstrap}.
\bpublisher{Springer}
(\byear{2012}).
\bid{doi={10.1007/978-1-4612-0795-5}, mr={1351010}}
\end{bbook}
%
\OrigBibText
Shao J., Tu D.: The Jackknife and Bootstrap. Springer-Verlag (2012)
\endOrigBibText
\bptok{structpyb}
\endbibitem

\bibitem{Seber}
\begin{bbook}
\bauthor{\bsnm{Seber}, \binits{G.}},
\bauthor{\bsnm{Lee}, \binits{A.}}:
\bbtitle{Linear regression analysis}.
\bpublisher{Wiley}
(\byear{2003}).
\bid{doi={\\10.1002/9780471722199}, mr={1958247}}
\end{bbook}
%
\OrigBibText
Seber G., Lee A.: Linear regression analysis. Wiley (2003)
\endOrigBibText
\bptok{structpyb}
\endbibitem

\bibitem{TG}
\begin{bbook}
\bauthor{\bsnm{Therneau Terry}, \binits{M.}},
\bauthor{\bsnm{Grambsch Patricia}, \binits{M.}}:
\bbtitle{Modeling Survival Data Extending the Cox Model}.
\bpublisher{Springer}
(\byear{2000}).
\bid{doi={10.1007/978-1-4757-3294-8}, mr={1774977}}
\end{bbook}
%
\OrigBibText
Therneau Terry M., Grambsch Patricia M.: Modeling Survival Data
Extending the Cox Model. Springer (2000)
\endOrigBibText
\bptok{structpyb}
\endbibitem

\bibitem{Wolter}
\begin{bbook}
\bauthor{\bsnm{Wolter}, \binits{K.}}:
\bbtitle{Introduction To Variance Estimation}.
\bpublisher{Springer}
(\byear{2007}).
\bid{mr={2288542}}
\end{bbook}
%
\OrigBibText
Wolter K.: Introduction To Variance Estimation. Springer (2007)
\endOrigBibText
\bptok{structpyb}%
\endbibitem

\end{thebibliography}
\end{document}